\documentclass[12pt]{amsart}

\usepackage[a4paper]{geometry}

\usepackage{amsfonts, amsmath, amssymb, amscd, latexsym, graphicx, pb-diagram, caption, mathdots}

\def\zz{\mathbb{Z}}

\def\rr{\mathbb{R}}

\newcommand{\ft}{\ensuremath{F_\tau}}
\def\xx{\bar x_0}
\def\xxx{\bar x_1}
\def\yy{\bar y_0}
\def\yyy{\bar y_1}

\def\qed{\hfill\rlap{$\sqcup$}$\sqcap$\par}

\newtheorem{thm}{Theorem}[section]
\newtheorem{prop}[thm]{Proposition}
\newtheorem{lem}[thm]{Lemma}

\newtheorem{defn}[thm]{Definition}

\newtheorem{question*}{Question}
\newtheorem{question}[thm]{Question}
\title{An irrational-slope Thompson's group}
\author{Jos\'e Burillo }
\address{Departament de Matem\`atiques,
Universitat Polit\`ecnica de Catalunya,
Jordi Girona 1-3,
08034 Barcelona, Spain}
\email{pep.burillo@upc.edu}
\thanks{The first author thanks the Spanish Ministry MICINN through grant MTM2017-82740-P for their support. The second and third author were supported by the Royal Society International Exchanges grant IES\textbackslash R3\textbackslash 170086.}

\author{Brita Nucinkis}
\address{Department of Mathematics, Royal Holloway, University of London,
Egham, TW20 0EX, UK}
\email{Brita.Nucinkis@rhul.ac.uk}

\author{Lawrence Reeves}
\address{School of Mathematics and Statistics,
University of Melbourne,
VIC 3010, Australia}
\email{lreeves@unimelb.edu.au}

\date{}

\begin{document}

\maketitle

\begin{abstract}
The purpose of this paper is to study the properties of the irrational-slope Thompson's group \ft\ introduced by Cleary in \cite{clearyirr}. We construct presentations, both finite and infinite, and we describe its combinatorial structure using binary trees. We show that its commutator group is simple. Finally, inspired by the case of Thompson's group $F$, we define a unique normal form for the elements of the group and study the metric properties for the elements based on this normal form. As a corollary, we see that several embeddings of $F$ in \ft\ are undistorted.
\end{abstract}

\section*{Introduction}

Thompson's groups were introduced in the 1960s and soon captured the interest of group theorists for their interesting properties. They have spawned a family of groups that have properties similar to the original Thompson's groups $F$, $T$ and $V$, but each of which is interesting in its own right.
The purpose of this paper is to study one of these groups, namely, the group \ft\ of piecewise linear homeomorphisms of $[0,1]$ having breakpoints in $\zz[\tau]$ and slopes that are powers of $\tau$, where $\tau$ is the golden ratio $(\sqrt5-1)/2$. The group $\ft,$ although having irrational breakpoints, does share many of the properties of the original Thompson's group $F.$

The group $\ft$ was introduced by Sean Cleary in \cite{clearyirr}, where it is first described and
proved to be of type $F_\infty$. The group $\ft=G([0,1]; \zz[\tau], \langle \tau \rangle)$ is also mentioned in the Bieri-Strebel notes \cite{bieristrebel}, although finite presentations there are only given for groups with rational slopes \cite[D.15.10]{bieristrebel}.

After defining the group we devote a section to  the combinatorial structure of the group which stems from the representation of elements by pairs of finite rooted trees.

In Section \ref{prez} a presentation is obtained:

{\bf Theorem 4.4}
$$\ft = \langle x_i, y_i \,|\, a_jb_i=b_ia_{j+1},    y_i^2=x_ix_{i+1} \mbox{ for } a,b \in \{x,y\} \mbox{ and } 0\leq i<j \rangle.$$

This infinite presentation can easily be reduced to a finite presentation which  is also  given.

The subsequent two sections describe the abelianisation of \ft, the commutator subgroup, and the main simplicity result.

{\bf Theorem 6.4} {\it The group $\ft'$ is simple.}

In Section \ref{Nform} we present an explicit normal form for the elements of the group:

{\bf Theorem 7.3} {\it Each element of \ft\ has a unique normal form representative.}

As with $F$, the unique normal form is closely related to a unique reduced tree diagram for the element, but here we need
 a new condition on the normal form to account for possible cancellations which can occur after performing a basic move (as introduced in Section \ref{2}).

The final two sections deal with metric considerations. As  happens with $F$, the metric on the group can be approximated by the number of carets in the unique reduced tree pair diagram. This gives us quasi-isometric embeddings, into $\ft,$ of some natural copies of $F$.

The reader is assumed to have some familiarity with Thompson's group $F$. Many of the arguments for \ft\ will be similar to those for $F$. In order to avoid repetitions and making this paper unnecessarily long, we will refer to the corresponding constructions and results for $F$ when necessary.  A good introduction for $F$, which contains many results which apply here, can be found in \cite{cfp}.

\section{Definition and first properties}

Let $\tau$ be the small golden ratio $\frac{\sqrt5-1}{2}=0.6180339887...$, which is a zero of the polynomial $X^2+X-1$. We will consider the ring $\zz[\tau]$ of elements of the type $a+b\tau$, where $a$ and $b$ are integers. Observe that $\tau=(1+\tau)^{-1}$ is a unit of this ring, hence we can consider  the group $G([0,1];\zz[\tau],\left<\tau\right>)$ of piecewise linear orientation preserving homeomorphisms of $[0,1]$ having breakpoints in the ring $\zz[\tau]$ and slopes in the subgroup $\langle \tau \rangle$ of units of this ring. Groups like these were introduced by  Bieri and Strebel \cite{bieristrebel}. Following Cleary, \cite{clearyirr}, we denote this group  \ft. Cleary proved that the group is of type $F_\infty$, so in particular, it is finitely presented, however no explicit finite presentation was given \cite{clearyirr}. Cleary also describes a combinatorial structure for $\ft$, which we are going to develop here, since it will be used extensively throughout this paper.

Observe that, since $1=\tau+\tau^2$, the unit interval can be subdivided into two subintervals of lengths $\tau$ and $\tau^2$ respectively and this can be done in two ways:
$
[0,1]=[0,\tau]\cup[\tau,1]$  {and}  $[0,1]=[0,\tau^2]\cup[\tau^2,1]$.
Since we also have that $\tau^k=\tau^{k+1}+\tau^{k+2}$ (for all $k\ge 0$), each subinterval can be further subdivided  into intervals whose lengths are powers of $\tau$.  Elements of $\ft$ are now given by a pair of such subdivisions into $n$ intervals each, together with an order preserving bijection, see \cite{clearyirr}.

This opens the door to a combinatorial approach to $\ft$  using binary rooted trees, with a caret representing a subdivision, in a very similar fashion to Thompson's group $F$. The difficulty here is that we will need two types of carets, since  intervals can be subdivided in two ways. Hence subdivisions will be represented by a caret with two edges of different lengths. 
In Figure \ref{subdunit} we have the two subdivisions of the unit interval given above, represented by their respective carets.
An example of an iterated subdivision with its corresponding tree is given in Figure \ref{tree}.
\begin{figure}[tbh]
\centerline{\includegraphics[width=80mm]{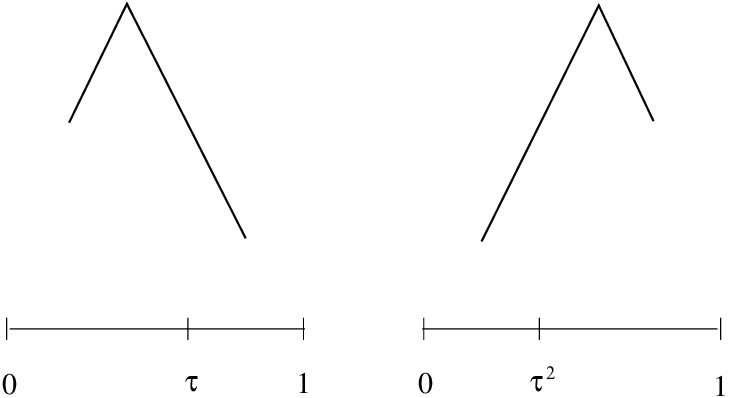}}
\caption{The two subdivisions of the unit interval and their carets.}
\label{subdunit}
\end{figure}

\begin{figure}[tbh]
\centerline{\includegraphics[width=80mm]{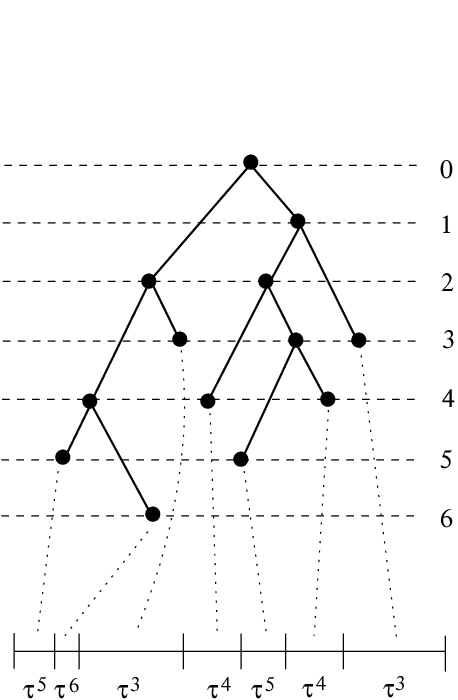}}
\caption{A tree with its corresponding subdivision and the nodes located in their corresponding levels.}
\label{tree}
\end{figure}

The reason for representing the shorter subinterval by the longer edge in the caret is that  then the nodes in the tree are organised by levels according to the lengths of the corresponding subintervals. A node at level $k$  corresponds to an interval of length $\tau^k$. In this way the tree carries more information than just the combinatorial structure of the intervals. See Figure \ref{tree} for an example.

\begin{defn} A caret with a long left edge and a short right edge is called an $x$-\emph{caret} or a caret of $x$-type, whereas the other type is called a $y$-\emph{caret}.
\end{defn}

The reason for this nomenclature will be apparent later, when we give a presentation of $\ft.$

\section{Combinatorics of the tree diagrams}\label{2}

An interesting feature of this group is that there are subdivisions which correspond to more than one tree.  The simplest example of this phenomenon is the subdivision of the unit interval into three subintervals of lengths $\tau^2,\tau^3,\tau^2$  given by
$
[0,\tau^2]\cup[\tau^2,\tau]\cup[\tau,1]
$.
 This can be represented by two trees, one with two $x$-carets and one with two $y$-carets, as shown in Figure \ref{basicmove}.
In Section 6 we produce a special unique tree pair diagram for each element of $\ft.$

\begin{figure}[tbh]
\centerline{\includegraphics[width=60mm]{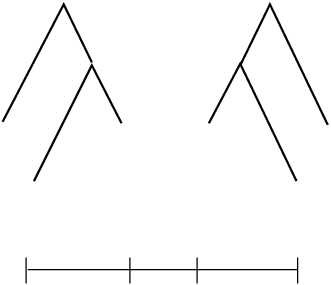}}
\caption{A subdivision of the unit interval with two trees that representing it giving rise to a basic move.}
\label{basicmove}
\end{figure}

The two trees in Figure \ref{basicmove} are crucial in what is to come. They are completely interchangeable when appearing, even as subtrees, as they represent the same subdivision of an interval.
We call the process of interchanging  these two configurations inside a tree a \emph{basic move}.  In Figure \ref{aftermove} we illustrate a basic move, represented by the thick lines, on the tree of Figure \ref{tree}.

\begin{figure}[tbh]
\centerline{\includegraphics[width=110mm]{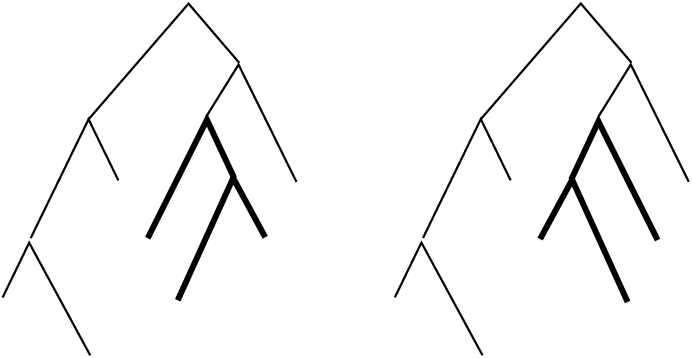}}
\caption{Performing a basic move on the tree from Figure \ref{tree}.}
\label{aftermove}
\end{figure}

 Clearly, an element can be represented by more than one tree pair diagram. Besides the usual phenomenon, familiar from $F$, where one can add or remove so-called redundant carets to obtain different tree pairs that represent the same element, here we can have two \emph{reduced} diagrams representing the same element. Recall that a reduced tree pair diagram is one without redundant carets. An example is given in Figure \ref{twodiags}, where a basic move on the right-hand-side diagram will produce redundant carets that can be removed to give the tree pair on the left.

\begin{figure}[tbh]
\centerline{\includegraphics[width=100mm]{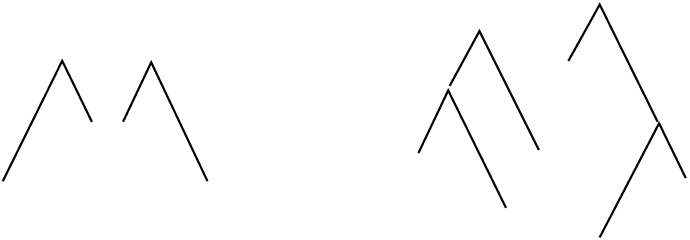}}
\caption{Two reduced tree pair diagrams representing the same element.}
\label{twodiags}
\end{figure}

Basic moves are crucial for working with the trees arising for this group. In particular, to be able to understand multiplication of elements via tree pair diagrams.  We begin by proving some important properties.

\begin{prop} \label{changecarets} By adding at most one caret to a tree $T$, any caret in $T$ can be switched from $x$-type to $y$-type or vice-versa.
\end{prop}

{\it Proof.} If the caret to be switched can have a basic move performed on it, then that switches the type. Suppose now that the short edge of this caret ends in a leaf (i.e., has no child), then add a caret of the same type to enable a basic move (see picture 2 of Figure \ref{common}). Finally, if the child caret on the short side is of the opposite type, then keep going down  short edges until a basic move can be performed. Should this process not result in a basic move, add a caret to the bottom short edge. Then one can perform multiple basic moves going back up to the caret to be switched. See pictures 3, 4 and 5 on Figure \ref{common}. \qed

The usefulness of basic moves is further illustrated by the following result, which will be used when deriving a presentation in Section \ref{prez}.

\begin{prop}\label{basicmoves}
Given two trees representing the same subdivision of the unit interval, one can always transform one into the other by a sequence of basic moves without adding additional carets.
\end{prop}

{\it Proof.} Let $T_1$ and $T_2$ be  two trees which represent the same subdivision, and assume that their root carets are different. Assume that $T_1$ has an $x$-type root caret and $T_2$ has a $y$-type root caret. Since $T_2$  has a $y$-caret at its root, the common subdivision  has a break at the point $\tau$ in the interval. This means that at $T_1$, the right edge (which is short) needs to be subdivided further, because  we need the break at $\tau$ on $T_1$ too. We are going to show that in order for the break at $\tau$ to show up in $T_1$, there have to be two consecutive carets of the same type somewhere on $T_1$. It follows that a basic move can be performed and the root caret of $T_1$ can be switched.

If the right child of the root caret in $T_1$  is of $x$-type, then we have two consecutive carets of the same type and a basic move could be performed at the root. 
Assume then that the short edge of the $x$-type root has a $y$-caret as child. Then the breaks are at the points $\tau^2$ and $2\tau^2=1-\tau^3$. Since
$
\tau^2<\tau<1-\tau^3
$, the desired breakpoint still has not been produced, see Figure \ref{fancyproof}.

The tree $T_1$ is therefore further subdivided. The point of the proof is that it is  necessary to have two consecutive subdivisions of the same type ($x$- or $y$- depending on the parity) to obtain a break at $\tau$. This is because of the following equalities (for even $n$, the odd case is similar):
$$
\tau=\tau^2+\tau^3
=\tau^2+\tau^4+\tau^5
=\tau^2+\tau^4+\tau^6+\tau^7
=\dots
=\displaystyle\sum_{k=1}^{n}\tau^{2k}+\tau^{2n+1}
$$
The odd power can only be produced with two consecutive carets of the same type, see Figure \ref{fancyproof2}.

\begin{figure}[tbh]
\centerline{\includegraphics[width=110mm]{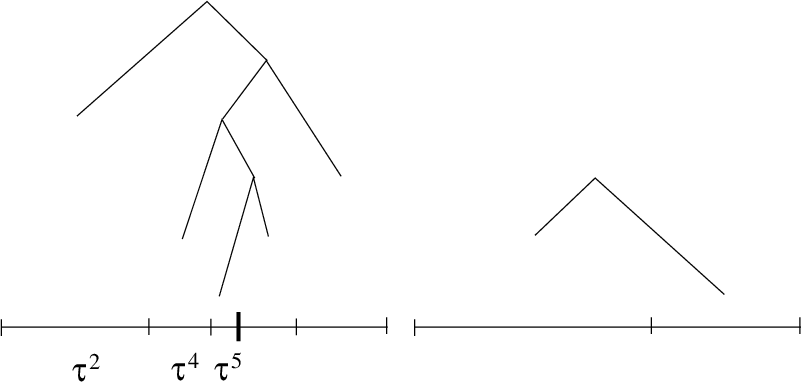}}
\caption{To have the break in $T_1$ we need two consecutive carets of the same type. Here two carets of $x$-type give the break on level 5 according to the equality $\tau=\tau^2+\tau^4+\tau^5$.}
\label{fancyproof2}
\end{figure}

\begin{figure}[tbh]
\centerline{\includegraphics[width=82mm]{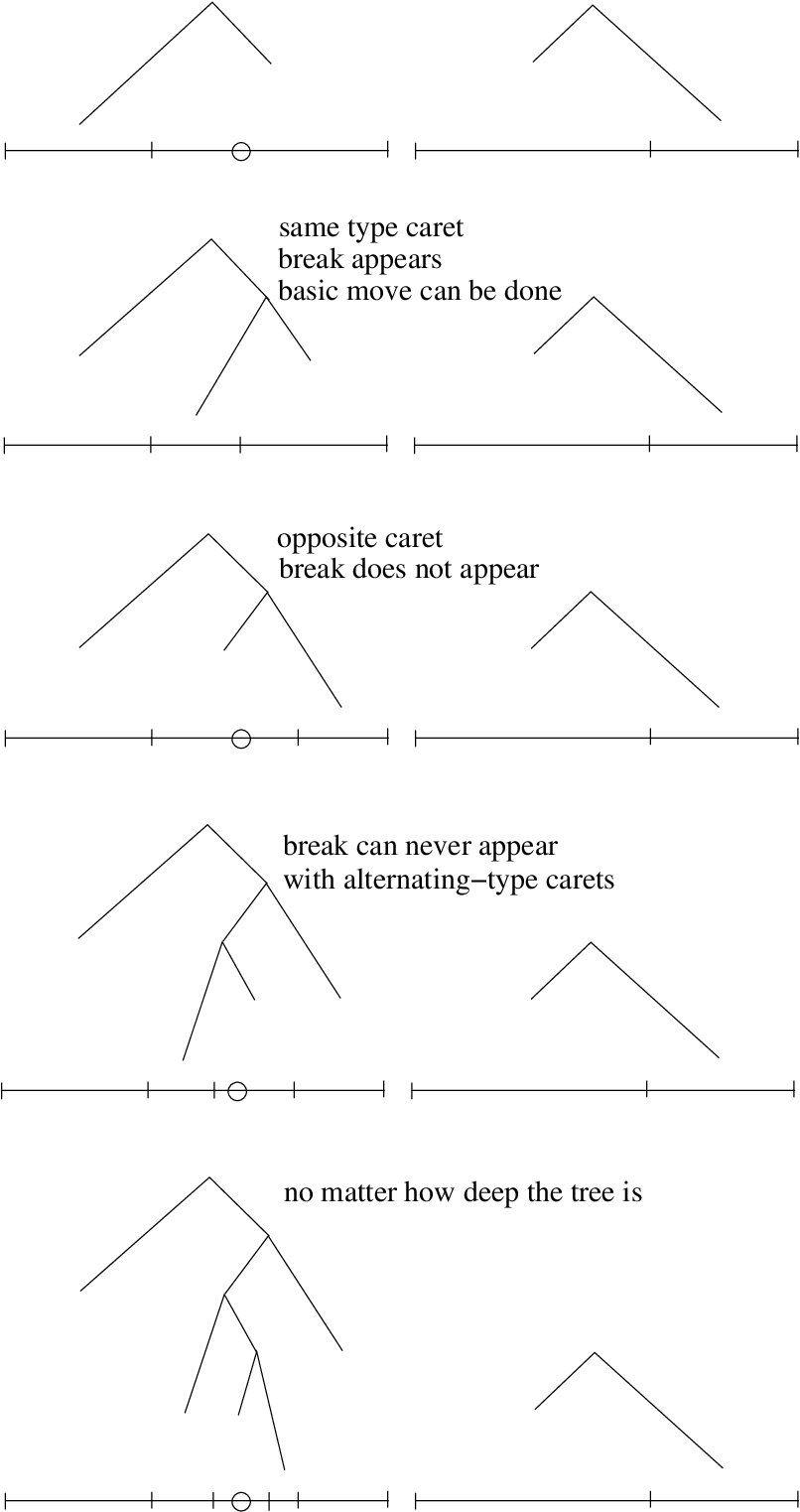}}
\caption{The subsequent subdivisions needed in tree $T_1$ to create the break at $\tau$. If the corresponding interval keeps being subdivided with the type different to the one immediately above, the break is never created. To create it, one needs to have two consecutive carets of the same type (as it happens in the second diagram in the figure. The circle represents the break we need to have because of $T_2$.}
\label{fancyproof}
\end{figure}

If below the short edge of the root caret the subsequent carets on their short edges are of alternating the type, we never reach the value $\tau$ because of the strict inequalities:
$$
\sum_{k=1}^n\tau^{2k}<\tau<1-\sum_{k=1}^n\tau^{2k+1}
$$
We would need infinitely many children to obtain a break at $\tau$ corresponding to the equalities
$$
\sum_{k=1}^\infty\tau^{2k}=\tau=1-\sum_{k=1}^\infty\tau^{2k+1}
$$
Figure \ref{fancyproof} illustrated this fact.

Since the trees are finite, in one of the trees there must be two consecutive carets of the same type, and the root caret can be switched by a sequence of basic moves without adding an extra caret.  Therefore we can keep going down the tree switching types of all the carets of different type, adding nothing, until the two trees are exactly the same.\qed

\section{Multiplication}

As in $F$, multiplication in  $\ft$ is given by composition of maps. To be able to multiply two elements given by tree pairs, we find a common expansion for the target tree of the first element and the source tree of the second element.  Consider two elements given by tree pairs $(T_1,T_2)$ and $(S_1, S_2)$. If it happens that $T_2=S_1$, then the product will be represented by $(T_1,S_2)$. If $T_1\neq T_2$ we apply the following proposition:

\begin{prop} \label{commontree} Given two trees $T$ and $T'$, there exists a common expansion tree $T''$, which represents a common subdivision of the subdivisions of $[0,1]$ given by $T$ and $T'$ respectively.
\end{prop}

{\it Proof.} If the carets are all of the same type, this can be done by just adding   carets to construct the least common expansion in the same way as is done with $F$. If both caret types are present we first need to switch the carets in $T$ to agree with those in $T'.$ This is done by applying Proposition \ref{changecarets}, see  Figure \ref{common} for an illustration. Once this process is finished we might need to add some more carets to obatin the tree $T'',$ which is an expansion of $T$ and has $T'$ as subtree.  \qed

\begin{figure}[tbh]
\centerline{\includegraphics[width=140mm]{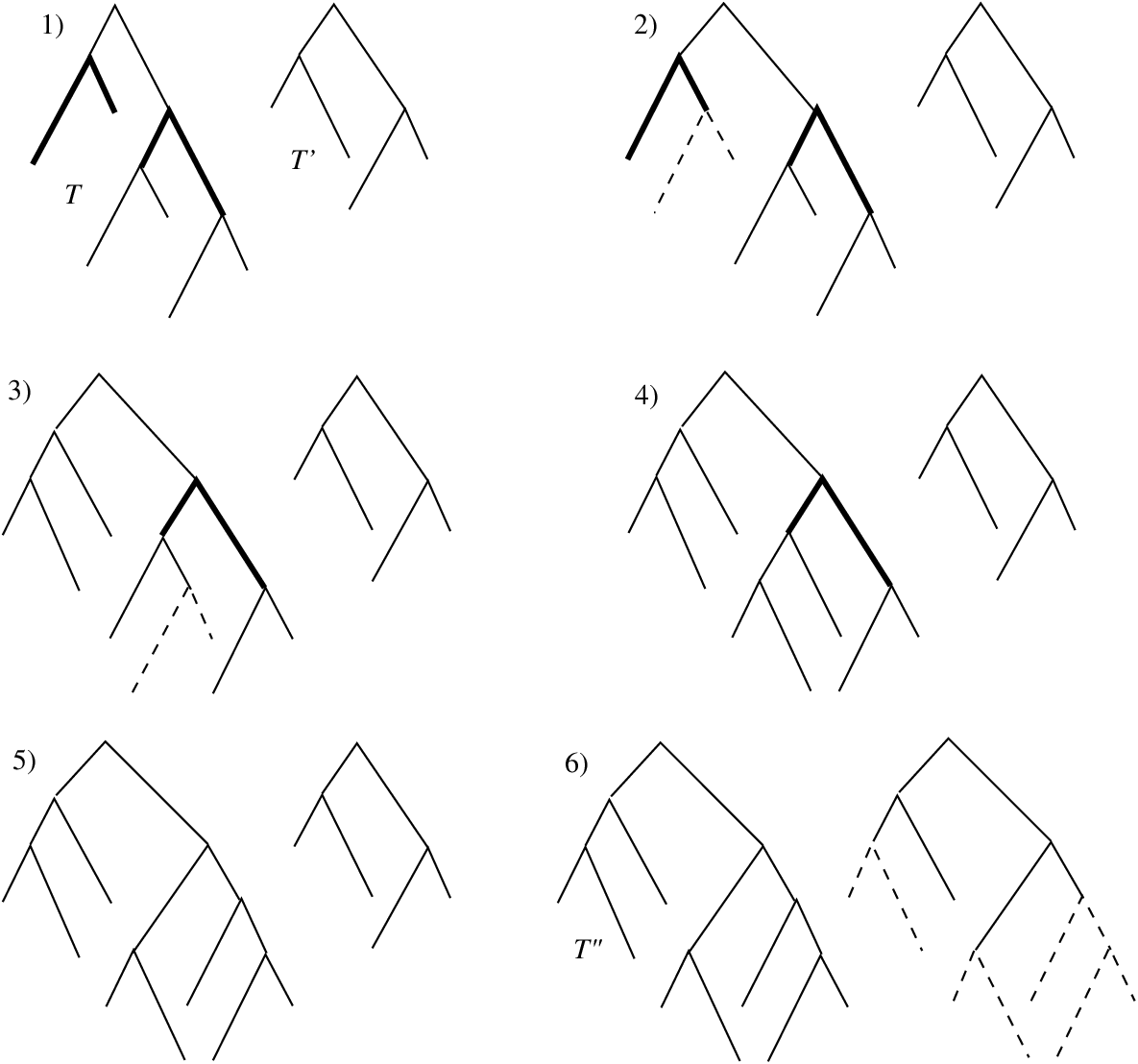}}
\caption{How to find a common subdivision for two trees $T$ and $T'$. Observe that the two carets in thick lines in $T$ are different from the corresponding ones in $T'$, so they will be switched using basic moves. We add a caret (in dashed lines, picture 2) so we can perform a basic move. For the second caret to be switched, a caret (also in dashed lines) needs to be added further below the caret we want to switch (picture 3), and two basic moves are required (pictures 4 and 5). Finally, once the first tree is a subdivision of the second one, we only need to add carets to the latter (tree $T''$, picture 6).}
\label{common}
\end{figure}

This procedure finishes the construction of the algorithm to perform the multiplication of two elements given by two tree pairs $(T_1,T_2)$ and $(S_1, S_2)$. Find the tree $T_3$ which is the common subdivision for $T_2$ and $S_1$, and find two tree pairs which represent the same elements which look like $(T_1',T_3)$ and $(T_3,S_2')$. This is done by adding to $T_1$ the carets corresponding to those we have added to $T_2$, and similarly for the other pair. Finally, the product is given by $(T_1',S_2')$. See Figure \ref{mult} for a simple example.

\begin{figure}[tbh]
\centerline{\includegraphics[width=70mm]{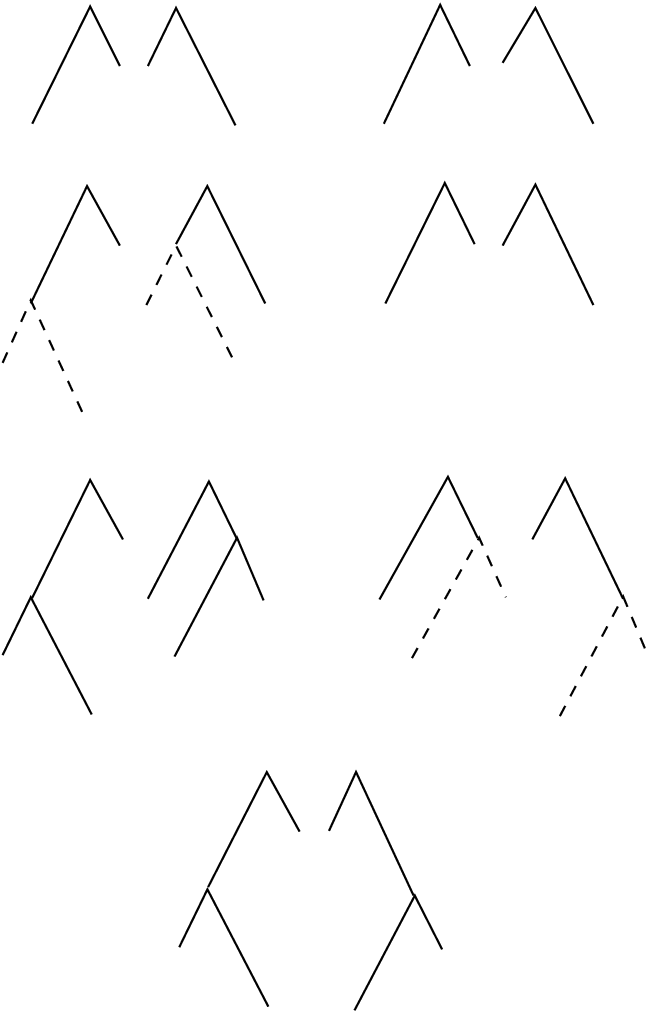}}
\caption{An example of how to multiply two elements when the corresponding carets do not coincide. In dashed lines the carets which are being added to be able to perform the multiplication.}
\label{mult}
\end{figure}

\section{Presentation}\label{prez}

To find generators for $F_\tau$ we follow the ideas used for $F$. The infinite generating set for $F$ has generators, which are given by binary tree pairs $(T,S)$, where $S$ is a tree where each caret has only right children, also called a spine, and where $T$ is a spine with one additional caret at the bottom left hand leaf. Since for $\ft$ we have two different kinds of carets, there is ambiguity in this construction.  However, by Lemma \ref{changecarets}, the type of a caret can always be switched.  Hence we chose one caret type for the trees we will call spines.

\begin{defn} A tree which has only \emph{right-side} carets of $x$-type is called a \emph{spine.}
\end{defn}

\begin{figure}[tbh]
\centerline{\includegraphics[width=30mm]{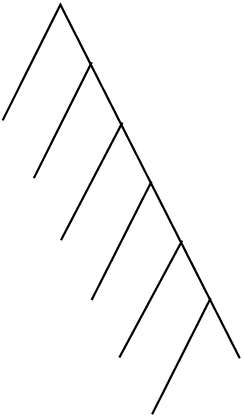}}
\caption{A spine.}
\label{spine}
\end{figure}

Generators for $F_\tau$ will have a spine to which an extra caret is added at the end, as a left child on the source tree, and as a right child on the target one.

\begin{defn} We define the  $x_n$ in $F_\tau$, for $n\geq 0$, by a tree pair diagram $(T_1,T_2),$ where $T_2$ is a spine  with $n+2$ carets, and where $T_1$ is a spine with $n+1$-carets together with in extra $x$-caret on the last left edge. Note that all carets in $x_n$ are $x$-carets, see Figure \ref{gensftau}. Similarly, we define elements  $y_n$ by having the same spine, but with the caret added to the source tree being of $y$-type.
\end{defn}

\begin{figure}[tbh]
\centerline{\includegraphics[width=70mm]{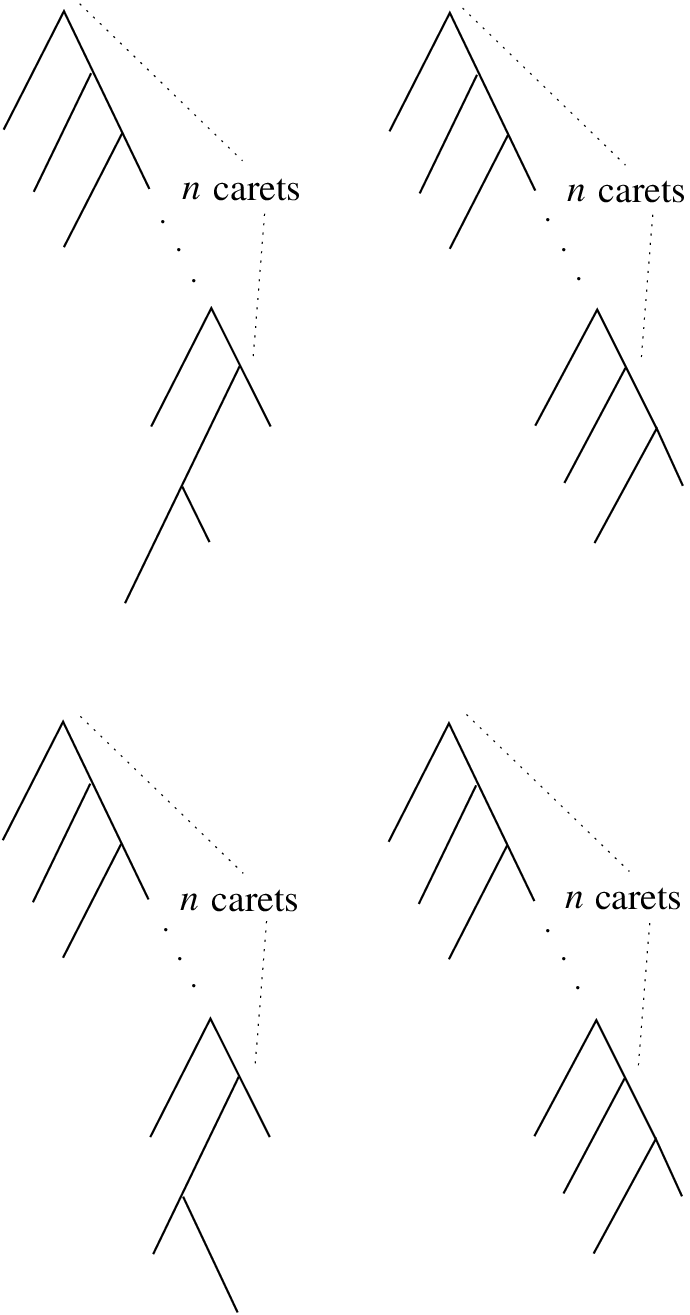}}
\caption{The generators $x_n$ and $y_n$ as tree diagrams.}
\label{gensftau}
\end{figure}

Observe from Figure \ref{gensftau} that the key caret is of type $x$ for the generators $x_n$ and of type $y$ for $y_n$, and this is the reason the carets are named in this way. However, in both cases, the spines have only $x$-carets.

Our goal is to prove that the set of $x_n$ and $y_n$, for $n\geq 0$, is a generating set for $F_\tau$. In a similar fashion to that for $F$, if the target tree of an element is a spine, this element is the product of generators (without taking inverses). In Figure \ref{example} we can see an example of an element which is the product of three generators, obtaining a pair made of a tree and a spine.

\begin{figure}[tbh]
\centerline{\includegraphics[width=80mm]{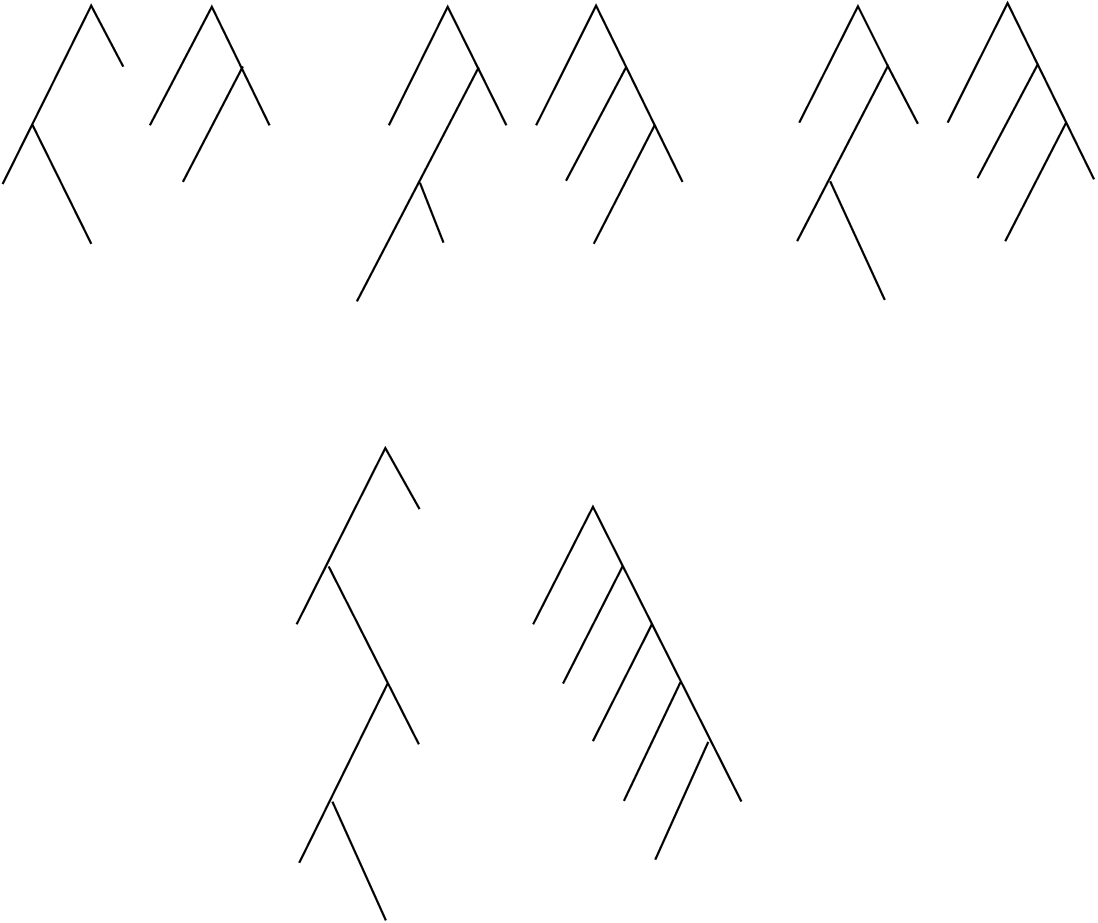}}
\caption{The element $y_0x_1y_1$ constructed as the product of the three generators.}
\label{example}
\end{figure}

In Figure \ref{positive} we see why any element given by a tree and a spine can be written as a product of the generators $x_n$ and $y_n$. If we multiply an element with a spine as a target tree by the generator $x_i$ or $y_i$, then the result is to attach a caret of the corresponding type to the $i$-th leaf. In this way we can construct a tree paired with a spine. Observe though that the tree constructed this way will always have in the right-hand side all carets of the $x$-type.

\begin{figure}[tbh]
\centerline{\includegraphics[width=80mm]{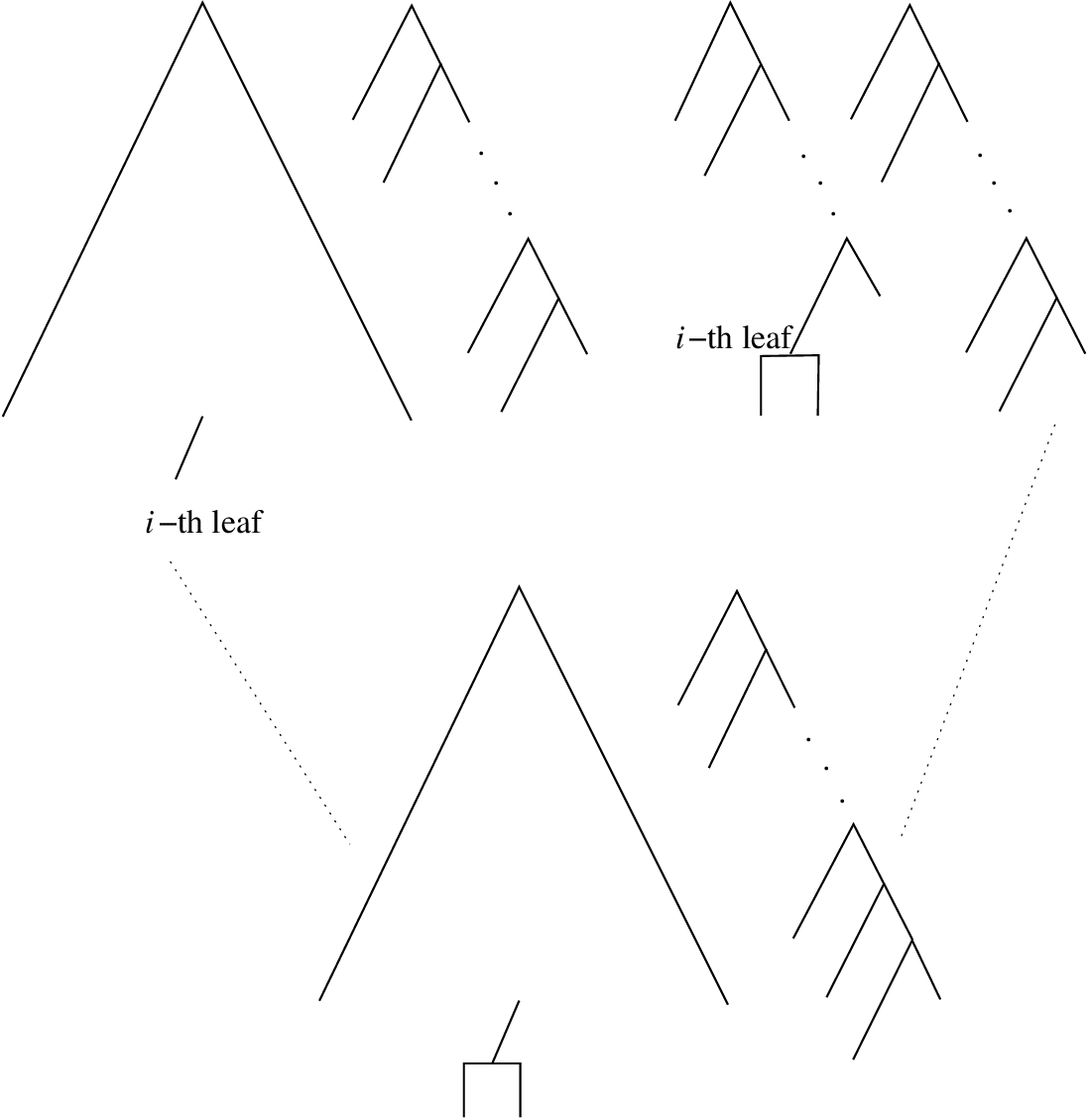}}
\caption{The multiplication of an element with a generator $x_i$ or $y_i$.}
\label{positive}
\end{figure}

Using this construction we can prove:

\begin{prop} The set of elements $x_n$ and $y_n$, for $n\geq 0$, is a set of generators for $F_\tau$.
\end{prop}

{\it Proof.} Take any element of $F_\tau$ given as a pair $(T_1,T_2)$ of trees. Using Lemma \ref{changecarets} we can assume that we have a spine down the right hand side of each tree. If the trees have an $y$-caret in the right-hand side, use the lemma to change the type of these carets, at the price of adding carets to the trees. The result will be a pair of trees whose right-hand sides have only $x$-carets.

Now, put a spine $S$ in between the two trees. The first tree pair $(T_1,S)$  gives an element which, by the construction specified above, is the product of generators $x_n$ or $y_n$. The second pair $(S,T_2)$ is the inverse of an element also of this type. Hence, any element is a product of the generators $x_n$ or $y_n$ and their inverses.\qed

It is not hard to see that there are some relations which are satisfied by these generators. The combinatorics of the carets, similar to those of $F$, give the following four sets of relators:
\begin{enumerate}
\item $x_jx_i=x_ix_{j+1}$
\item $x_jy_i=y_ix_{j+1}$
\item $y_jx_i=x_iy_{j+1}$
\item $y_jy_i=y_iy_{j+1}$
\end{enumerate}
where in all cases we have $i<j$. Another  set of relators is given by the subdivision which admits two expressions as carets. 
These relations are $y_n^2=x_nx_{n+1}$. The goal of the next theorem is to show that these are all relations needed to obtain a presentation for $\ft$.

\begin{thm} \label{presentation} A presentation for $\ft$ is given by the generators $x_i,y_i$, with the relations
\begin{enumerate}
\item $x_jx_i=x_ix_{j+1}$
\item $x_jy_i=y_ix_{j+1}$
\item $y_jx_i=x_iy_{j+1}$
\item $y_jy_i=y_iy_{j+1}$
\item $y_i^2=x_ix_{i+1}$
\end{enumerate}
for $0\leq i<j$.
\end{thm}

{\it Proof:} Given a word in the generators $x_i,y_i$ which gives the identity, when we construct its corresponding tree pair diagram, the two trees necessarily give rise to the same subdivision. Also, the two trees will have a spine (all $x$-carets) in their right hand sides. According to Proposition \ref{basicmoves}, we can go from one to the other by applying basic moves to one of them, and in this case, the basic moves are never performed on a vertex on the right hand side of the tree. Observe that each basic move corresponds exactly to multiplying our word by a conjugate of relation (5), noting that all instances of relation (5) have spines and hence are precisely those that we need. Hence, using relation (5) we can obtain a word which yields a diagram where the two trees are the same. In the  same way as  is done for Thompson's group $F$, this diagram can be seen to be a consequence of relations of the type (1) to (4).
Hence, the original word is a consequence of the relations (1) to (5).\qed

This presentation allows us to establish a correspondence between tree pair diagrams and a particular type of words. This correspondence is completely analogous to that in Thompson's group $F$, based on \emph{leaf exponents}. See \cite[Theorem 2.5]{cfp} or \cite[Section 3.1]{thompt}. Observe that the relations (1)--(4) allow for the ordering of the generators in a word by index, increasing for positive powers and decreasing for negative powers. We have the following result.

\begin{prop}\label{leafexponent} Any element of $F_\tau$ admits an expression of the type
$$
a_{i_1}a_{i_2}\dots a_{i_n}b_{j_m}^{-1}\dots b_{j_2}^{-1}b_{j_1}^{-1}
$$
where
\begin{enumerate}
\item the letters $a$ and $b$ represent either $x$ or $y$,
\item $i_1\leq i_2\leq\ldots\leq i_n$ and $j_1\leq j_2\leq\ldots\leq j_m$.
\end{enumerate}
\end{prop}

This is analogous to the normal form for Thompson's group $F$. This expression for an element corresponds to its tree pair diagram using leaf exponents. The only difference between $\ft$ and $F$ is that we can alternate generators $x_i$ and $y_i$ within the same index, as in the example at Figure \ref{example}, where we consider the element $y_0x_1y_1$.

It is not difficult to deduce a finite presentation from the infinite one. From the relations (1)--(4) it is easily seen that the generators with index 2 or higher are conjugates of those with index 1. Hence, the only generators needed are $x_0,x_1,y_0,y_1$. Similarly, for each family of relators (1) to (4), only two are needed
 as  happens in Thompson's group $F$, see, for instance, \cite{cfp}. For the family (5), observe that if $i\geq 2$, the relation $y_i^2=x_ix_{i+1}$ is a conjugate (by the appropriate power of $x_0$) of $y_1^2=x_1x_2$. Hence, the following relations are sufficient:
$$ \arraycolsep=1.5pt
\begin{array}{rclrcl}
x_2x_1&=&x_1x_3& \rule{10mm}{0mm}x_3x_1&=&x_1x_4\\
x_2y_1&=&y_1x_3&x_3y_1&=&y_1x_4\\
y_2x_1&=&x_1y_3&y_3x_1&=&x_1y_4\\
y_2y_1&=&y_1y_3&y_3y_1&=&y_1y_4\\
y_0^2&=&x_0x_1&y_1^2&=&x_1x_2
\end{array}
$$
We do not claim that this presentation is optimal, and it is possible that there is a presentation with fewer generators or with fewer relations.

\section{Abelianisation and the commutator subgroup}

Once we have a presentation, it is easy to abelianise the group. The abelianised group has four generators $\xx,\xxx,\yy,\yyy$, and observe that the relations (1)-(4) abelianise trivially. Hence the quotient abelian group has two relations, namely
$$
2\yy=\xx+\xxx\qquad2\yyy=2\xxx
$$
where we have  changed to additive notation for the abelian group. Using the first relation we can eliminate the generator $\xx$.  Defining   $\bar z=\xxx-\yyy$,  the abelianisation can be   generated by $\xxx,\yy,\bar z$ subject to the relation $2\bar z=0$. The abelianisation is therefore isomorphic to $\zz^2\oplus\zz_2$.

The commutator subgroup, that is the kernel of the abelianisation map, can also be completely understood. Looking at the generators $\xxx$ and $\yy$, we see that they represent the slopes at 0 and at 1, in the same way as holds for Thompson's group $F$. The map from $\ft$ to $\zz^2$ given by the two components of the abelianisation map generated by $\xxx$ and $\yy$ coincides (up to a change of basis in $\zz^2$) with the map that sends every element to the  slopes at 0 and 1.

\begin{defn} We say that an element $f\in\ft$ has \emph{support bounded away from 0 and 1,} or simply \emph{bounded support}, if there exists $\varepsilon>0$ such that $f$ is the identity in the intervals $[0,\varepsilon]$ and $[1-\varepsilon,1]$. We will denote the subgroup of elements with bounded support by $\ft^c$.
\end{defn}

Observe then that the commutator subgroup is contained in $\ft^c$. However, from  the $\zz_2$ component we see that it is not equal to it. To describe it clearly, let $z=x_1y_1^{-1}$, and observe that $z$ maps to  $\bar z$ under the abelianisation map.

\begin{prop} The commutator subgroup of $\ft$ is formed exactly by those elements in $\ft^c$ such that the total exponents in $x_1$ and $y_1$ are both even. Equivalently, they are the elements in $\ft^c$ which have even exponent for $z$, that is, which abelianise to zero on the $\zz_2$ component.
\end{prop}

The proof is elementary by looking at the interpretations of the abelianisation given above.

According to Proposition \ref{leafexponent} and the corresponding word-diagram, the extra condition for an element to be in $\ft'$ (the total exponents in $\xxx$ and $\yyy$ in the abelianisation are both even) can be read off the diagram. Recall that a binary tree has left, right and interior carets according to their location in the tree. Left carets are on the left side of the tree, each of them connected to the root by a chain of left children. Right carets are connected to the root by right children, and interior carets are those carets that are neither left or right, see, for instance, \cite{blakegd} or \cite{growth}.

Let a diagram have the trees $T_1$ and $T_2$. We can identify the total exponent for $\xxx$ and $\yyy$ according to the number of carets in the diagram. Define the following numbers, for $i=1,2$:

\begin{itemize}
\item Let $n_i$ be the number of interior $x$-carets in $T_i$.
\item Let $m_i$ be the number of interior $y$-carets in $T_i$.
\item Let $r_i$ be the number of left $x$-carets in $T_i$.
\item Let $s_i$ be the number of left $y$-carets in $T_i$.
\end{itemize}


\begin{prop} When the element given by a diagram $(T_1,T_2)$ is abelianised, the component for $\xxx$ is $n_1-r_1-n_2+r_2$ . Also, the total number of $\yyy$ is $m_1-m_2$.
\end{prop}

Furthermore, although we shall have no need for it, the number $s_1-s_2$ gives the component for $\yy$. 

Observe that to obtain the total exponent for $\xxx$ one has to take the $n_i$, which correspond to the generators $x_i, i\geq 1$, but also the $r_i$, which correspond to the generators $x_0$ in the word. This is because in the abelianisation, each $\xx$ is replaced by $-\xxx+2\yy$. Hence, each $\xx$ contributes with a unit to the exponent for $\xxx$. Observe that we will only be interested in parity, so we can discard all  negative signs.

Hence, just for looking at the tree pair diagram we can know if an element is in the commutator subgroup or not.

\begin{thm}\label{commutator-thm} An element given by a diagram $(T_1,T_2)$ is in $\ft'$ if and only if the following conditions are all satisfied:
\begin{enumerate}
\item The level of the leftmost leaf is the same for $T_1$ and $T_2$, that is, $2r_1+s_1=2r_2+s_2$. This corresponds to the fact that the element must be the identity in a neighbourhood of 0.
\item The level of the rightmost leaf is the same for $T_1$ and $T_2$. This corresponds to the fact that the element must be the identity in a neighbourhood of 1.
\item The total exponents for $\xxx$ and $\yyy$ are both even, i.e. $n_1+r_1+n_2+r_2$ and $m_1+m_2$ are both even.
\end{enumerate}
\end{thm}

This interpretation will be useful in the next section.

\section{Simplicity}

The goal of this section is to prove that the commutator subgroup $\ft'$ is a simple group. The proof will follow several steps.

\begin{defn}
Let $a,b\in\zz[\tau]$, with $0<a<b<1$. Then we denote by $\ft[a,b]$ the subgroup of $\ft$ of those elements whose support is included in $[a,b]$. Within $\ft[a,b]$, we denote by $\ft'[a,b]$ its commutator subgroup and also $\ft^c[a,b]$ its subgroup of elements with bounded support (i.e. they are the identity in a neighbourhood of $a$ and in one of $b$). For clarity, observe that the support of an element in $\ft^c[a,b]$ is included in $[a+\varepsilon,b-\varepsilon]$ for some $\varepsilon>0$.
\end{defn}

We have that for any $a,b$ the subgroup $\ft[a,b]$ is isomorphic to $\ft$.

\begin{prop}\label{int} $\ft[a,b]\cong\ft$.
\end{prop}

{\it Proof.} According to \cite[Corollary 1]{clearyirr}, there exists an element $f\in\ft$ such that $f(\tau^2)=a$ and $f(\tau)=b$. Conjugating by $f$, we see that $\ft[a,b]\cong\ft[\tau^2,\tau]$. To see that $\ft[\tau^2,\tau]$ is isomorphic to $\ft$, we only need to scale the maps by a factor of $\tau^3$, which is the length of $[\tau^2,\tau]$, and observe that $\tau$ is a unit of the ring $\zz[\tau]$. \qed

Since the support of the elements of $\ft[\tau^2,\tau]$ is contained in $[\tau^2,\tau]$, we can represent these by a tree pair diagram given  by the 2-caret spine which appears at the root, and with the rest of the diagram hanging only from the middle leaf of this 2-caret spine, see Figure \ref{tau2tau}.

Now we will look at the commutator subgroup of $\ft[a,b]$.

\begin{lem}\label{ftauab} Let $f\in\ft'$, and let $a,b\in\zz[\tau]$ be that $f$ is  in $\ft^c[a,b]$. Then, $f\in\ft'[a,b]$.
\end{lem}

Observe that from the fact that $f\in\ft'$ we cannot immediately conclude that $f$ is in $\ft'[a,b]$. The extra condition of Theorem \ref{commutator-thm} for $\xxx$ and $\yyy$ refers to the generators of $\ft$ and not to those of $\ft[a,b]$. We need to relate the generators of both groups to be able to establish the result.

{\it Proof.} As in the  proof of Proposition \ref{int}, we can assume that $a=\tau^2$ and $b=\tau$. The element $f$ has bounded support in $\ft[\tau^2,\tau]$, but in order for it to be in $\ft'[\tau^2,\tau]$ it must satisfy the conditions of Theorem \ref{commutator-thm} with respect to the generators of $\ft[\tau^2,\tau]$. Let $\phi:\ft\rightarrow\ft[\tau^2,\tau]$ be the isomorphism described above, that is, the map obtained by hanging trees from the middle leaf of a two-caret spine. Then, $\ft[\tau^2,\tau]$ is generated by $\phi(x_0),\phi(x_1),\phi(y_0),\phi(y_1)$. See Figure \ref{tau2tau} to clarify this situation.

\begin{figure}[tbh]
\centerline{\includegraphics[width=110mm]{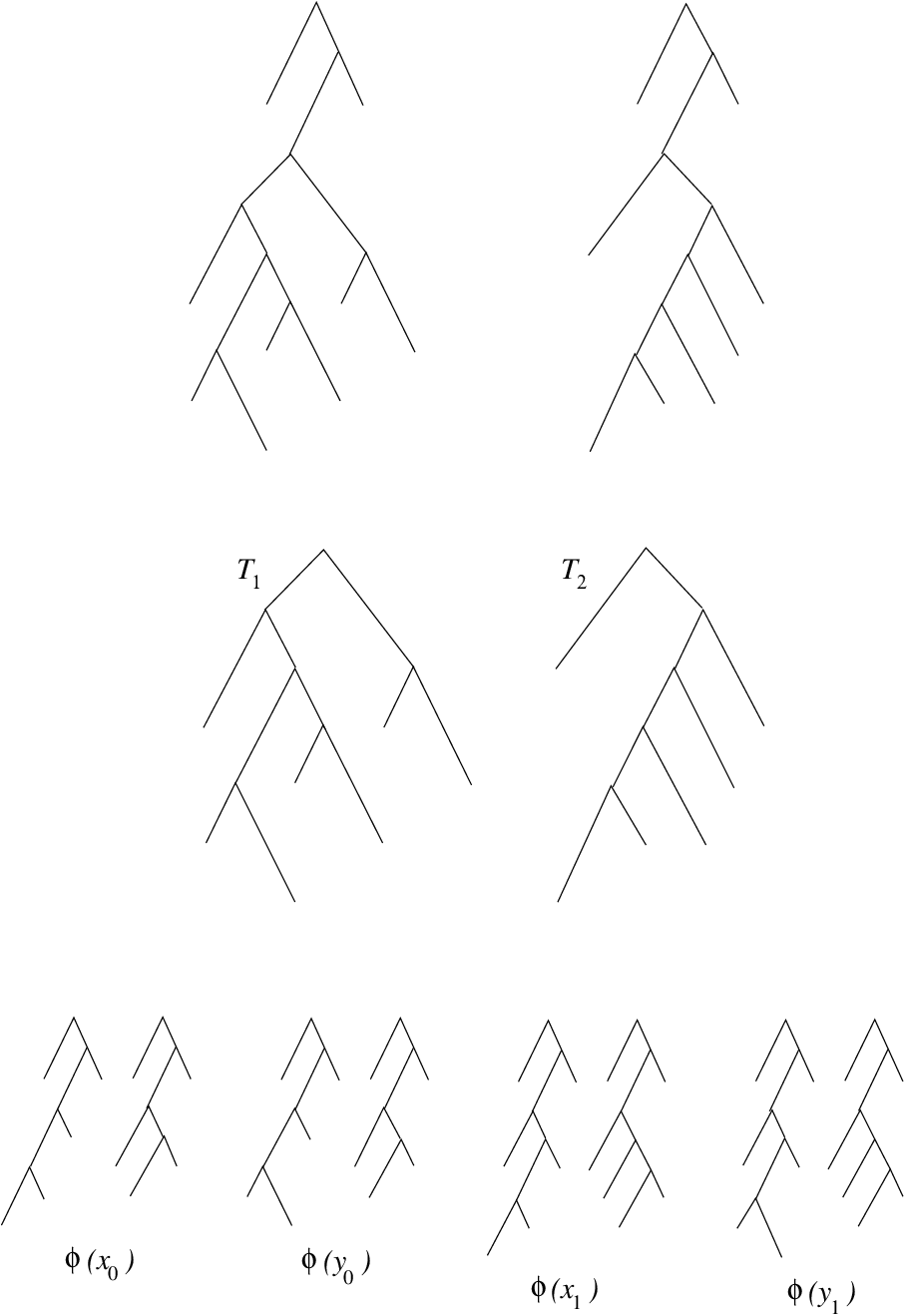}}
\caption{The interpretation of the subgroup $\ft[\tau^2,\tau]$ in terms of diagrams. The top diagram represents an element of $\ft[\tau^2,\tau]$, because the two carets located at the root indicate that the whole support is included in the interval $[\tau^2,\tau]$. The bottom row represents the generators of $\ft[\tau^2,\tau]$, and the top diagram is the image under $\phi$ of the second diagram.}
\label{tau2tau}
\end{figure}

Consider the two trees $(T_1,T_2)$ such that the diagram for $f$ is obtained by attaching $T_1$ and $T_2$ to the middle leaf of a 2-caret spine, as we did above. Let $n_i,m_i,r_i,s_i$ be the numbers of right and interior $x$-type and $y$-type carets as defined in the previous section. We know that $f$ is in $\ft'$ and we want to show that $f\in\ft'[\tau^2,\tau]$. But the diagram for $f$ when considered in $\ft[\tau^2,\tau]$ would be $(T_1,T_2)$, whereas the diagram for $f$ when considered in $\ft$ has the trees $T_1$ and $T_2$ attached to a two-caret spine.

Hence, the number of $\xxx$ and $\yyy$ for $f$ in $\ft$ has to consider all the carets in $T_1$ and $T_2$ as interior carets, since they hang from the middle leaf in a two-caret spine. This means that by being in $\ft'$ we know that the numbers $n_1+r_1+n_2+r_2$ and $m_1+s_1+m_2+s_2$ are even (observe that the right carets in $T_1$ and $T_2$, which are now interior, are the same number in both trees so their number is always even). And to see that $f\in\ft'[\tau^2,\tau]$ we need that the numbers which have to be even are now $n_1+r_1+n_2+r_2$ and $m_1+m_2$. The first of these numbers is the same in both cases, and for $m_1+m_2$ we only need to see that since $f$ is the identity in a neighbourhood of $\tau^2$, we have that $2r_1+s_1=2r_2+s_2$ and then $s_1-s_2$, and hence $s_1+s_2$, is even. So from this and from $m_1+s_1+m_2+s_2$ being even, we conclude that $m_1+m_2$ is even and hence $f\in\ft'[\tau^2,\tau]$.\qed

We can now state and prove the main theorem.

\begin{thm}\label{simple} The group $\ft'$ is simple.
\end{thm}

The proof will occupy the remainder of this section. It will be based on the following theorem due to Higman. Let $\Gamma$ be a group of bijections of some set $E$. For $g\in \Gamma$ define its \emph{moved set} $D(g)$ as the set of points $x\in E$ such that $g(x)\ne x$. This is analogous to the support, but since \emph{a priori} there is no topology on $E$, we do not take the closure.

\begin{thm}[Higman]\label{Higman}
Suppose that for all $\alpha, \beta, \gamma\in \Gamma\setminus \{1_{\Gamma}\}$, there is a $\rho\in \Gamma$ such that
$\gamma(\rho(S))\cap \rho(S)=\varnothing$ where $S=D(\alpha)\cup D(\beta)$.
Then the commutator subgroup $\Gamma'$ is simple.
\end{thm}

The proof can be seen in \cite{higmansimple}.

The idea of using this theorem is to take advantage of the high transitivity of Thompson-like groups to see that they easily fulfill the conditions of Higman's theorem. As we have already used before,  \cite[Corollary 1]{clearyirr} implies that  given two closed intervals $A$ and $B$, such that $0,1\notin A$, there exists an element of $\ft$ such that $f(A)\subset B$. Hence, the conditions of Higman's theorem are easily seen to be satisfied. Since $\gamma\neq 1$ it is easy to find an interval $C$ such that $\gamma(C)\cap C=\varnothing$. Also, use transitivity to find $\rho$ to send $S$ inside $C$.

The only thing is that the condition $0,1\notin A$ means that Higman's theorem cannot be applied to $\ft$, because there are many elements whose support is the whole unit interval. We can  apply Higman's theorem to the commutator $\ft'$, because all its elements have bounded support. The conclusion of the application of Higman's theorem is then that the \emph{second} commutator $\ft''$ is simple. The proof of Theorem \ref{simple} will be finished when we prove the following lemma.

\begin{lem}\label{firstsecond} We have that $\ft'=\ft''$.
\end{lem}

{\it Proof:} Clearly we have that $\ft''\subset\ft'$. For the reverse inclusion, take $f\in\ft'$. Since $f\in\ft^c$, choose $a,b\in\zz[\tau]$ such that $f\in\ft^c[a,b]$, namely, if the support of $f$ is included in the interval $[c,d]$, choose $a,b$ satisfying $0<a<c<d<b<1$. According to Lemma \ref{ftauab}, we have that $f\in\ft'[a,b]$. Hence, we have that $f=[p_1,q_1][p_2,q_2]\ldots[p_k,q_k]$, where $p_i,q_i\in\ft[a,b]\subset\ft^c$. To finish the proof and see that $f\in\ft''$, it would be enough to prove that $p_i,q_i$ are in $\ft'$, but this need not be true. We will modify these elements to get the desired result.

Observe that $p_i,q_i$ have support in $[a,b]$, but we have no information on whether they have an even or odd number of generators $\bar z$ when abelianised. For $p_i,q_i$ to be in $\ft'$ we  need each of them to have an even number of generators $\bar z$.

Observe that the element $z=x_1y_1^{-1}$ has bounded support, that is, it is in $\ft^c$. Choose now a tiny interval $[u,v]$ such that $[a,b]\cap[u,v]=\varnothing$. This means that either $0<u<v<a$ or $b<u<v<1$, either one works. As we have done before, and according to  \cite[Corollary]{clearyirr}, choose an element $g\in\ft$ which maps the support of $z$ inside $[u,v]$. Let $z'$ be the conjugate of $z$ by $g$ in such a way that the support of $z'$ is now inside $[u,v]$. Finally, since $[u,v]$ is disjoint with $[a,b]$, we have that $z'$ commutes with each of the $p_i,q_i$, for all $i=1,\ldots,k$. Therefore, we have that for each $i=1,\ldots,k$,
$$
[p_i,q_i]=[p_iz',q_i]=[p_i,q_iz']=[p_iz',q_iz']
$$
Since $z'$ is a conjugate of $z$, it contributes exactly  one generator $\bar z$ to the abelianisation. Hence, for each $i$, exactly one of these four commutators has both terms with an even number of $\bar z$. For instance, if $p_1$ has an odd number of generators $\bar z$ and $q_1$ has an even number, the commutator we choose to have two elements with even $\bar z$ will be $[p_1z',q_1]$.

By choosing the appropriate commutator for each $i$, we can get all $k$ commutators to have two terms with even number of $\bar z$, and hence we conclude that all terms involved in all commutators are in $\ft'$, and from this, finally, that $f\in\ft''$. \qed

This lemma, together with Higman's theorem applied to $\ft'$, implies that $\ft'$ is simple.

\section{Normal Form}\label{Nform}

In this section we describe a normal form (with uniqueness) for $F_\tau$ that is very similar to that for $F$.
A word over the $x_i$ and $y_i$ will be said to be in \emph{seminormal form} if it has the following form:
$$
x_0^{a_0}y_0^{\epsilon_0}x_1^{a_1}y_1^{\epsilon_1}\dots x_n^{a_n}y_n^{\epsilon_n} x_m^{-b_m}x_{m-1}^{-b_{m-1}}\dots x_1^{-b_1}x_0^{-b_0}
$$
where $a_i,b_i\ge 0$ and $\epsilon_i\in\{0,1\}$. Observe that $y$ generators only appear in the positive part of the word, and that they are only allowed to have exponents zero or one. From the correspondence between diagrams and words on the generators via leaf exponents described in Section 3, the existence of a seminormal form for each element of \ft\ follows from the following result.
\begin{lem}\label{movingcarets} Let $(S_1,S_2)$ be a tree pair with $S_i$ having only $x$-carets down its right spine.
There exists a tree pair  $(T_1,T_2)$ representing the same element of \ft\
 that satisfies the following:
\begin{enumerate}
\item $T_2$ has no $y$-carets,
\item $y$-carets in $T_1$ have no left children.
\end{enumerate}
Moreover, the number of carets in $T_i$ is bounded above by three times the number of carets in $S_i$.
\end{lem}
\begin{proof}
As noted previously all elements of \ft\ have tree pairs in which
$T_1$ and $T_2$ have only $x$-carets on their right side.
Starting with such a pair,
we first modify $T_2$ so that all $y$-carets in $T_2$ have no left children.
This is done by working from right to left as follows.
Suppose a $y$-caret is such that it has left children, but that all $y$-carets of higher leaf index do not have left children.
 We want to swap the type of the caret.
If the immediate left child is also a $y$-caret, then we perform a basic move.
Suppose then that we have an $y$-caret whose left child is an $x$-caret. There are three possibilities determined by the right child of the $x$-caret.
These are illustrated in Figure \ref{LS1}. Note that in the third case in the picture, the bottom $y$-caret has a higher leaf index than the top one, so it must have no left child.
In each case, after adding at most one caret, the original $y$-caret can be moved down closer to the leaf.
\begin{figure}[tbh]
\centerline{\includegraphics[width=70mm]{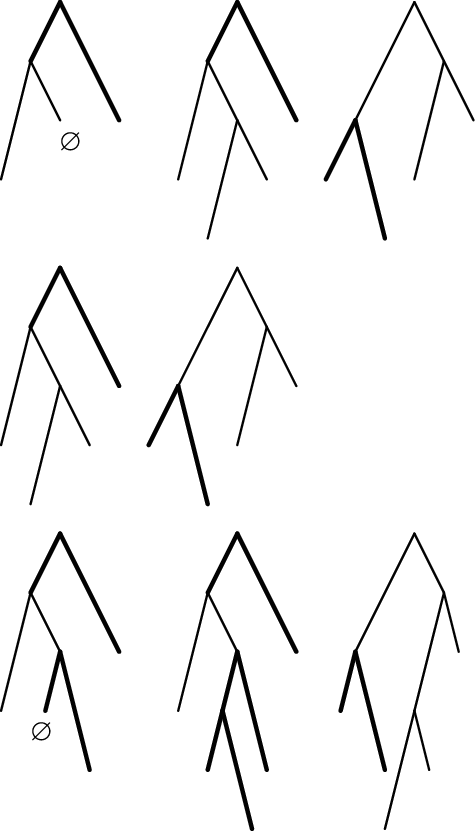}}
\caption{ \label{LS1} Changing the type of a $y$-caret. The $y$-carets are bold. In the first case, the indicated leaf is subdivided by the addition of an $x$-caret and then two basic moves are carried out.
In the second case no subdivision is needed. In the third situation a $y$-caret is added and then a sequence of three basic moves applied. Note that in each case the original (topmost) $y$-caret has been moved down the tree.}
\end{figure}
The new tree $T_2$ now has the property that each $y$-caret has no left child.
Now for each $y$-caret in $T_2$ add another $y$-caret as the left child and perform a basic move. The resulting tree $T_2$ now has no $y$-carets.

Following the same process as above, we can move $y$-carets in $T_1$ down the tree to ensure that $T_1$ satisfies (2).
We need to be careful not to add any $y$-carets to $T_2$. To that end we modify the third case to be that shown in Figure \ref{LS2}, adding two $x$-carets instead of a $y$-caret.
\begin{figure}[tbh]
\centerline{\includegraphics[width=70mm]{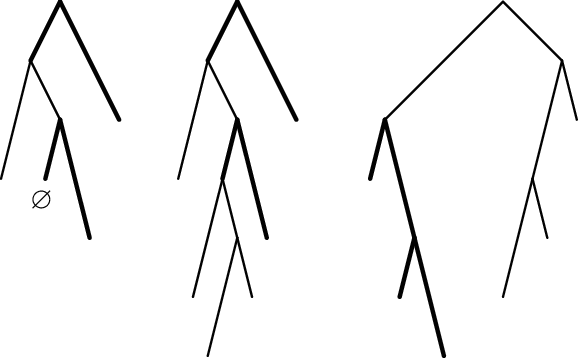}}
\caption{ \label{LS2} Rather than adding a single $y$-caret, two $x$-carets are added and then a sequence of four basic moves results in the final tree. Notice that the topmost $y$-caret has been moved down the tree while preserving the fact that the lower $y$-caret has no left child.}
\end{figure}
Notice that  given a tree pair $(S_1,S_2)$, the above proof produces a tree pair $(T_1,T_2)$ satisfying the conclusion of the  lemma and such that the number of carets
added is at most twice the original number of carets.
\end{proof}

Two different words, each in seminormal form, can represent the same element of \ft. This can happen in two ways. First, we can have a reduction similar to that seen in $F$, where Thompson relators can be applied to reduce the subscripts of many generators in the word. This corresponds to a diagram being nonreduced and the removal of exposed matching carets. The second way this can happen is more subtle and corresponds to an example such as $x_0y_0x_2x_1^{-1}x_0^{-1}=y_0$. Both these words are in seminormal form and both are reduced, but after performing a couple of basic moves, two carets become exposed and they can be cancelled. See Figure \ref{normform}. This will be called a \emph{hidden cancellation.} Fortunately, the only possible hidden cancellations will be exactly of this type. A hidden cancellation shows up every time we have a subword of the form $x_iy_ix_{i+2}ux_{i+1}^{-1}x_i^{-1}$ where $u$ is a word involving generators of index at least $i+3$. If that happens, we have the following sequence of equalities using relators:
$$
x_iy_ix_{i+2}ux_{i+1}^{-1}x_i^{-1}=x_ix_{i+1}y_iux_{i+1}^{-1}x_i^{-1}
=y_i^3ux_{i+1}^{-1}x_i^{-1}
=y_ix_ix_{i+1}ux_{i+1}^{-1}x_i^{-1}
=y_iu'
$$
where $u'$ is the same word as $u$ but with all subscripts lowered by 2.

These two types of reductions are the only possible obstructions for the uniqueness of the seminormal form, as we will show next. We define a normal form as a word which is not allowed to have any of these possible reductions.

\begin{defn}\label{normalform}
A word $w$ is said to be in \emph{normal form} if it is in seminormal form and, in addition, for all $i$ we have:
\begin{enumerate}
\item If $a_i$ and $b_i$ are both nonzero, then at least one of $a_{i+1}, b_{i+1}, \epsilon_{i}, \epsilon_{i+1}$ is nonzero.
\item  If $w$ contains a subword of the form $x_iy_ix_{i+2}ux_{i+1}^{-1}x_i^{-1}$, then $u$ contains a generator with index either $i+1$ or $i+2$.
\end{enumerate}
\end{defn}
As previously noted, these conditions are best understood in terms of tree pair diagrams. The first condition, as for $F$, corresponds to matching exposed carets that can be eliminated. The second condition corresponds to a situation in which two basic moves result in matching exposed carets. This is illustrated in Figure \ref{normform}.
\begin{figure}[tbh]
\centerline{\includegraphics[width=80mm]{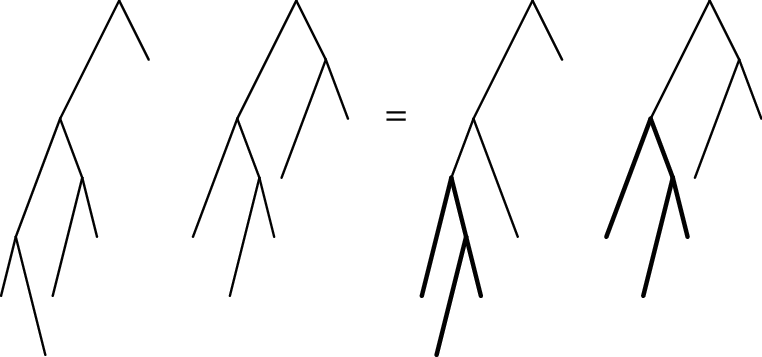}}
\caption{A hidden cancellation. \label{normform} The tree pair on the left is reduced and corresponds to $x_0y_0x_2x_1^{-1}x_0^{-1}$. After performing two basic moves, the two carets in bold can be cancelled. The diagram we obtain is $y_0$. }
\end{figure}

\begin{thm}
Each element of \ft\ has a unique normal form representative.
\end{thm}
\begin{proof}
 Familiarity with the proof of uniqueness of the normal form for $F$ (as shown, for instance, in \cite{bg}) will be of great help understanding this proof.

That each element of \ft\ has a representative word in normal form is straightforward. The first four relations as listed in Theorem \ref{presentation} can be used, as with $F$, to have the indices in increasing order in the positive part and decreasing in the negative one. Then use Lemma \ref{movingcarets} to transform this word into seminormal form. If this word then fails either of the conditions in the definition of the normal form, then there is a strictly shorter representative in seminormal form. Keep reducing the word until both conditions are satisfied. Equivalently, keep reducing the diagram both for exposed matching carets and for hidden cancellations.

For uniqueness, consider two normal form words $u$ and $v$  that represent the same element of \ft\, and have minimum total length among all such pairs in the whole group. Let the words be given by
$$
u\equiv x_0^{a_0}y_0^{\epsilon_0}u_1x_0^{-b_0}
\qquad
v\equiv x_0^{c_0}y_0^{\zeta_0}v_1x_0^{-d_0}
$$
where $a_0,b_0,c_0,d_0\ge 0$, $\epsilon_0,\zeta_0\in\{0,1\}$ and $u_1$ and $v_1$ are normal form words in which all subscripts are at least 1. The symbol $\equiv$ is being used to denote equality as words.
We will assume that not all $a_0,b_0,c_0,d_0,\epsilon_0,\zeta_0$ are zero. If this were not the case then the following argument can be readily modified by moving to the least subscript for which this is true, but we keep the case of zero for simplicity and clarity.
%

Equating the slopes at 0 for the piecewise linear maps determined by $u$ and $v$ we obtain
$2a_0+\epsilon_0-2b_0=2c_0+\zeta_0-2d_0$, from which it follows that
\begin{equation}\label{slope}
a_0-b_0=c_0-d_0\qquad\text{and}\qquad\epsilon_0=\zeta_0. \tag{$\ast$}
\end{equation}
Since $u$ and $v$ were chosen to minimise the total length we have that one of $a_0$ and $c_0$ must be zero, or else an $x_0$ could be cancelled to obtain shorter words. Similarly, one of $b_0$ and $d_0$ must be zero. We can assume that $c_0=0$. It then follows from (\ref{slope}) that $d_0=0$ and $a_0=b_0\neq 0$. We deal separately with the two possible cases: $\epsilon_0=0$ and $\epsilon_0=1$, which will correspond to conditions (1) and (2) of the normal form, respectively.

In the case in which $\epsilon_0=0$ we move the generators $x_0$ from $u$ to $v$ so we have $u_1=x_0^{-a_0}v_1x_0^{a_0}=v_2$ where $v_2$ is the word obtained from $v_1$ after increasing all subscripts by $a_0$. The word $v_2$ is in normal form and all subscripts appearing in it are 2 or more. Since $u_1=v_2$, both words are in normal form and the total length is strictly less than that for the original pair, we conclude that $u_1\equiv v_2$. But then the original word $u\equiv x_0^{a_0}v_2x_0^{-a_0}$ would have violated condition (1) in Definition \ref{normalform}.

Suppose now that $\epsilon_0=1$. Our words are now $u\equiv x_0^{a_0}y_0u_1x_0^{-a_0}$ and $v\equiv y_0v_1$. We move one generator $x_0$ from each side of $u$ to $v$, so we have
\begin{align*}
x_0^{a_0-1}y_0u_1x_0^{-(a_0-1)}&=x_0^{-1}y_0v_1x_0
=x_0^{-1}y_0x_0v_2
=x_0^{-1}y_0x_0x_1x_1^{-1}v_2\\
&=x_0^{-1}y_0^3x_1^{-1}v_2=x_1y_0v_3x_1^{-1}
=y_0x_2v_3x_1^{-1}
\end{align*}
where $v_2$ is the word obtained by increasing all subscripts in $v_1$ by 1 (by moving the $x_0$ across it), and then subsequently $v_3$ also obtained from $v_2$ increasing the subscripts while moving $x_1^{-1}$ across. Notice that all indices for $v_3$ are at least 3 and hence the final word is still in normal form.  Observe too that the total length for these two words is exactly the same as the original pair, because we have added two generators and later eliminated two. Now repeating the above for all pairs of $x_0$ until they are exhausted, we obtain
$$
y_0u_1=y_0x_2v'x_1^{-1}
$$
where $v'$ is a normal form word in which all subscripts are at least 3, and the length is still the same as the original one. After cancelling the $y_0$, and since $u_1$ and $x_2v'x_1^{-1}$ are in normal form, we conclude, by the minimality of the original pair, that
$u_1\equiv x_2v'x_1^{-1}$. But then the original word
$u\equiv x_0^{a_0}y_0x_2v'x_1^{-1}x_0^{-a_0}$ would not have satisfied part (2) of the definition of the normal form, having a forbidden subword with all subscripts for $v'$ being at least 3.
\end{proof}

\section{Metric properties}

Once we have a unique normal form for the elements of \ft, we can compute some estimates for the word metric of elements, based on the normal form and the unique reduced diagram that relates to it. The idea and the procedures are very similar to those for $F$, see \cite{burillo} and \cite{bcs}.

Given an element $g\in\ft$, take its normal form
$$
g=x_0^{a_0}y_0^{\epsilon_0}x_1^{a_1}y_1^{\epsilon_1}\dots x_n^{a_n}y_n^{\epsilon_n} x_m^{-b_m}x_{m-1}^{-b_{m-1}}\dots x_1^{-b_1}x_0^{-b_0}
$$
where both $a_n+\epsilon_n$ and $b_m$ are nonzero (i.e. we have a positive generator of index $n$ and a negative one of index $m$), and there are no cancellations between $x_n$ and $x_m^{-1}$ (i.e. either $\epsilon_n=1$ or else $n\neq m$).
\begin{defn}
We define the number
$$
D(g)=a_0+a_1+\dots+a_m + \epsilon_0+\epsilon_1+\dots+\epsilon_n+ b_0+b_1+\dots+b_m + n + m
$$
and we denote by $N(g)$ the number of carets of either tree of the unique diagram which corresponds to the normal form, that is, a reduced diagram with no hidden cancellations.
\end{defn}

These two quantities are good estimates for the word metric.

\begin{thm} There exists a constant $C>0$ such that we have
$$
\dfrac{D(g)}C\le\| g\|\le C\,D(g)\qquad\text{ and }
\qquad\dfrac{N(g)}C\le\|g\|\le C\,N(g)
$$
where $\|g\|$ represents the word metric with respect to the generating set $x_0,x_1,y_0,y_1$.
\end{thm}

{\it Proof.} Since each $x$ or $y$ generator is represented by a caret, we have the obvious inequalities:
\begin{align*}
N(g)&\ge a_0+a_1+\cdots+a_n+\epsilon_0+\epsilon_1+\cdots+\epsilon_n\\
N(g)&\ge b_0+b_1+\cdots+b_m\\
N(g)&\ge n\\
 N(g)&\ge m
\end{align*}
which yield the inequality $D(g)\leq 4\,N(g)$.

For the upper bounds, take each generator $x_i$ and $y_i$ with $i\geq 2$ and rewrite it in terms of  $x_0,x_1,y_0,y_1$ to obtain the desired bound. The positive part of the word can be written as
$$
x_0^{a_0}y_0^{\epsilon_0}x_1^{a_1}y_1^{\epsilon_1}x_0^{-1}x_1^{a_2}y_1^{\epsilon_2}x_0^{-1}\dots x_0^{-1}x_1^{a_n}y_1^{\epsilon_n}x_0^{n-1}
$$
because observe that a sequence $\ldots x_i^{a_i}y_i^{\epsilon_i}x_{i+1}^{a_{i+1}}\ldots$ will have a large number of generators $x_0$ cancelled in between:
$$
\begin{aligned}
\ldots x_i^{a_i}y_i^{\epsilon_i}x_{i+1}^{a_{i+1}}\ldots&=\ldots(x_0^{-i+1}x_1^{a_i}x_0^{i-1})
(x_0^{-i+1}y_1^{\epsilon_i}x_0^{i-1})(x_0^{-i}x_1^{a_{i+1}}x_0^{a_i})\ldots\\&=
\ldots x_0^{-1}x_1^{a_i}y_1^{\epsilon_i}x_0^{-1}x_1^{a_{i+1}}\ldots
\end{aligned}
$$
and hence for the word we only have one generator $x_0^{-1}$ every time the index grows by 1. Similarly, we do the same for the negative part. Clearly then, we have that the length of this word in $x_0,x_1,y_0,y_1$ is, for instance, at most $2\,D(g)$.

For the lower bound, we use the number of carets. Start with a shortest word for an element $g$, with length $L=\|g\|$. When multiplying by a generator, observe that since a generator has at most three carets, the number of carets of the diagram can increase by at most three carets, plus possibly added carets needed to perform the multiplication. But a generator has only one caret which is not on the spine, and since the spine has only $x$-carets all the time, only one caret may need to be changed to multiply and then only one caret may have to be added. Hence, when we multiply by a generator the number of carets can grow by at most four. From the shortest word we can then obtain  a diagram which has at most $4L$ carets.

This diagram will have $x$ and $y$ generators mixed in each index (see Proposition \ref{leafexponent}), so it has to be modified so that only one $y$-caret appears for each index and with no left children, according to Lemma \ref{movingcarets}. We observe carefully the process described in that proof, and as  is indicated there, the number of carets can at most triple, because we may need to add two carets for each original one. Hence, the total number of carets of the diagram corresponding to the seminormal form is at most $12L$. Reducing and eliminating hidden cancellations can only decrease the number of carets. From here we have that $\|g\|\geq N(g)/12$. Summarizing all inequalities, we have
$$
\frac{D(g)}{48}\leq\frac{N(g)}{12}\leq\|g\|\leq 2\,D(g)\leq 8\,N(g)
$$
and this finishes the proof.\qed

\section{Distortion}

The similarities between the metric properties of \ft\ and $F$ allow us to state some distortion results for subgroups in \ft\ which are isomorphic to $F$.

Diagrams in \ft\ may have two types of carets. We can consider the subgroup of \ft\ of those elements which can be written with only one type. But if only one type (say $x$) of carets is used, then the combinatorics are exactly those of $F$, and the subgroup is obviously isomorphic to $F$. We will call $F_x$ the copy of $F$ inside \ft\ given by elements with a diagram containing  only $x$-carets. Observe that this subgroup is  generated by  the $x_i$ generators, or, if we prefer, generated by $x_0$ and $x_1$, clearly yielding a copy of $F$ inside \ft. We have the following result:

\begin{prop}\label{fxundist} The inclusion of $F_x$ inside \ft\ is undistorted.
\end{prop}

{\it Proof.} If an element of $F$ has a diagram (with regular equally-sided carets), then the same diagram but now with $x$-carets will give a reduced diagram for \ft. Observe also that the normal form in $F$ is also a normal form in \ft. Hence the number of carets is the same for both groups. Since in both cases the number of carets is equivalent to the word metric, we obtain the desired result.\qed

The $y$-sided counterpart of this result is a bit more complicated. We can clearly consider the subgroup of \ft\ generated by $y_0$ and $y_1$. This subgroup is also clearly isomorphic to $F$, for instance observe that the combinatorics of the diagrams are exactly the same, but diagrams here have $x$-carets on the spine and $y$-carets everywhere else. Hence, due to the bias we have chosen for the generators (and hence the normal forms) having $x$-careted spines, this subgroup is \emph{not} the same as the subgroup of \ft\ with \emph{all} carets of $y$-type. This latter subgroup will be called $F_z$ and it is generated by the following two elements:
$$
z_0=y_0y_2\qquad \mbox{ and } \qquad z_1=y_2y_4,
$$
see Figure \ref{F_z}. This is a proper subgroup of $F_y$, and is also isomorphic to $F$. But both these subgroups behave well.

\begin{figure}[tbh]
\centerline{\includegraphics[width=130mm]{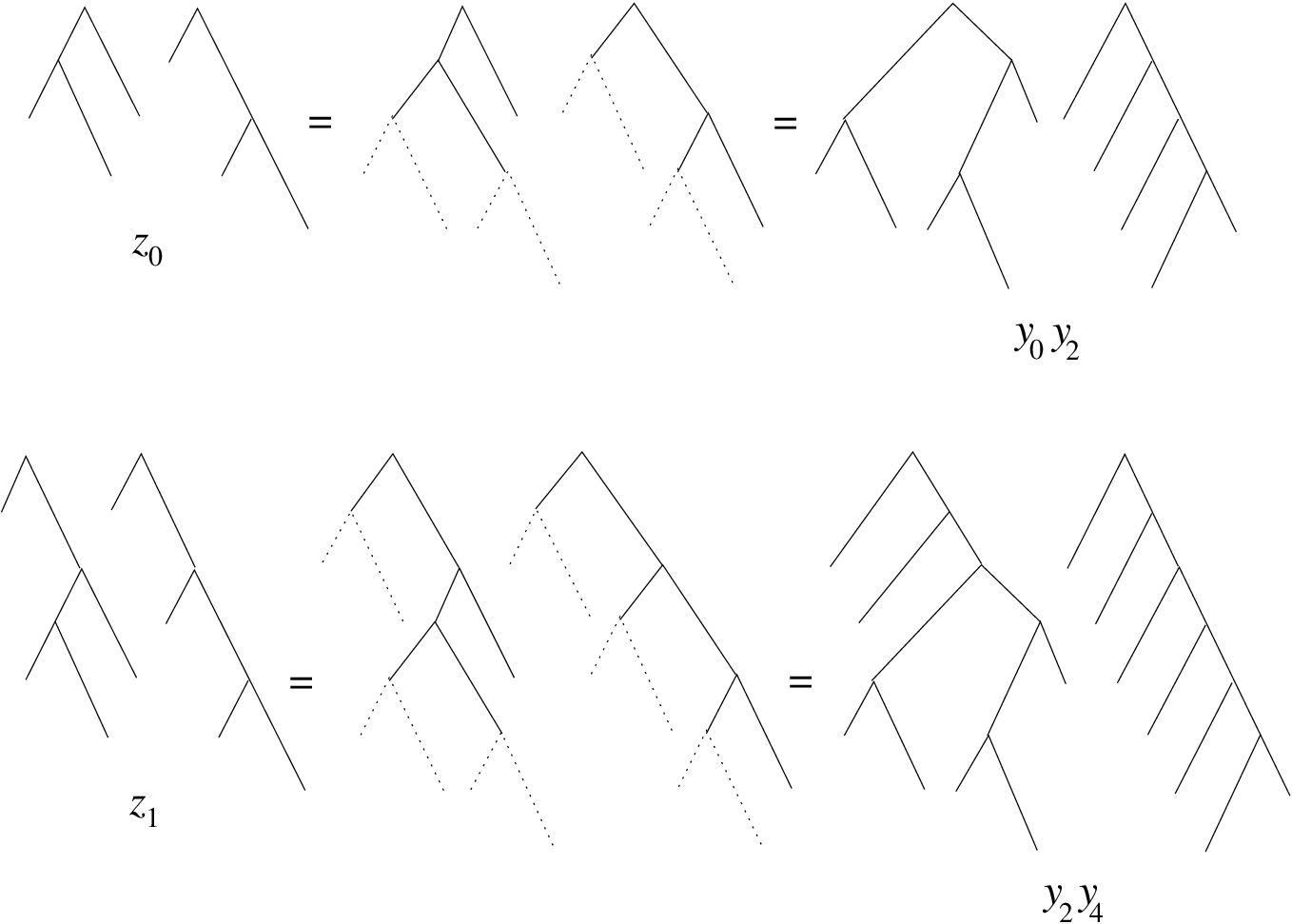}}
\caption{ \label{F_z} The generators of $F_z$ transformed into elements with $x$-careted spines. Actually these are their normal forms. Originally they only have $y$-carets, but their expressions in the $y$ generators (and hence their normal forms) need to have $x$-carets on the spine.}
\end{figure}

\begin{prop} The inclusions of $F_y$ and $F_z$ inside \ft\ are both undistorted.
\end{prop}

{\it Proof.} The case of $F_z$ of elements with only $y$-carets is actually symmetric to $F_x$.  If instead we had chosen spines consisting of $y$-carets, and the generating set for $\ft$ by adding carets to this $y$-spine, we would have ended up with exactly the same situation as in Proposition \ref{fxundist} and the result for $F_x$ follows immediately.

For the subgroup $F_y$, elements here have $x$-carets in the spine and $y$-carets in the interior and left side of the trees. To compute the number of carets of their normal forms most of the $y$-carets have to be transformed into $x$-carets (except a few at the bottom with no left children), but as we have seen in Lemma \ref{movingcarets}, the number of carets can at most triple in this process. Hence the number of carets in $F$ and in \ft\ differ by a multiplicative constant, so the distances do too, and the inclusion is undistorted.\qed

It is interesting to remark that in previous examples of groups of the Thompson family where two different types of carets appear, copies of $F$ inside which use only one type of caret were always distorted. See \cite{wladisdist} and \cite{2vm}. Hence, \ft\ is the first known example of a group of the Thompson family whose elements have two types of carets but whose $F$ subgroups of a single type of caret are undistorted.

\section{Conclusions and future directions}

The properties of this group which are different from those of $F$ arise from the special type of carets and their relation. The basic move provides the new relation not seen before, which in turn causes torsion in the abelianisation. Furthermore, previous examples where we have two different types of carets  (Thompson--Stein groups, higher dimensional Thompson's groups) have distorted copies of $F$ inside, due to the fact that these carets do not merge well and lead to a fast growth of the number of carets. Here, due to the relation and the basic move, carets can have their type easily changed and hence their number does not grow. This is the reason why $F$ is undistorted in \ft.

The original motivation to study this group was the  question by Brin of whether every finitely presented group of piecewise-linear homeomorphisms of $\rr$ could be found as a subgroup of $F$. This is the reason the torsion in the abelianisation was considered, since it was not known whether a subgroup of $F$ could have torsion in the abelianisation. A finitely generated subgroup of $F$ was found whose abelianisation contains 2-torsion during the Oberwolfach workshop 1823b \emph{Cohomological and Metric Properties of Groups of Homeomorphisms of $\rr$} (see \cite{ober}). The authors would like to thank the participants of the workshop for very fruitful discussions. However, the following is still open:

\begin{question} Can a \emph{finitely presented} subgroup of $F$ can have torsion in its abelianisation?
\end{question}

 Brin's question has been answered by Lodha \cite{lodha_f23}, where it is proved that the Thompson-Stein group $F_{2,3}$ is not a subgroup of $F$. However, this group has non-cyclic slope group. Hence we believe that the following is still interesting:

 \begin{question} Does $\ft$ embed in $F$?
 \end{question}

Many of the properties for \ft\ are also present in the groups $T_{\tau}$ and $V_{\tau}$, the $T$ and $V$ versions of \ft. For instance, we can still perform basic moves and also have a copy of $\zz_2$ in the abelianisation. Hence, these groups are no longer simple, but  both have an index-two subgroup which is simple. These ideas have been developed in \cite{ttau}, which is the natural continuation of this paper.

\bibliography{pepsrefs}

\def\cprime{$'$}
\begin{thebibliography}{10}

\bibitem{bieristrebel}
Robert Bieri and Ralph Strebel.
\newblock On groups of {PL}-homeomorphisms of the real line.
\newblock arXiv:1411.2868.

\bibitem{bg}
Kenneth~S. Brown and Ross Geoghegan.
\newblock An infinite-dimensional torsion-free {${\rm FP}\sb{\infty }$} group.
\newblock {\em Invent. Math.}, 77(2):367--381, 1984.

\bibitem{burillo}
Jos{\'e} Burillo.
\newblock Quasi-isometrically embedded subgroups of {T}hompson's group ${F}$.
\newblock {\em J. Algebra}, 212(1):65--78, 1999.

\bibitem{growth}
Jos{\'e} Burillo.
\newblock Growth of positive words in {T}hompson's group {$F$}.
\newblock {\em Comm. Algebra}, 32(8):3087--3094, 2004.

\bibitem{ober}
Jos\'e Burillo, Kai-Uwe Bux, and Brita Nucinkis.
\newblock Cohomological and metric properties of groups of homeomorphisms of
  $\rr$.
\newblock Oberwolfach report for workshop 1823b, to appear.

\bibitem{2vm}
Jos{\'e} Burillo and Sean Cleary.
\newblock Metric properties of higher-dimensional {T}hompson's groups.
\newblock {\em Pacific J. Math.}, 248(1):49--62, 2010.

\bibitem{bcs}
Jos\'e Burillo, Sean Cleary, and Melanie Stein.
\newblock Metrics and embeddings of generalizations of {T}hompson's group
  ${F}$.
\newblock {\em Trans. Amer. Math. Soc.}, 353(4):1677--1689 (electronic), 2001.

\bibitem{thompt}
Jos{\'e} Burillo, Sean Cleary, Melanie Stein, and Jennifer Taback.
\newblock Combinatorial and metric properties of {T}hompson's group {$T$}.
\newblock {\em Trans. Amer. Math. Soc.}, 361(2):631--652, 2009.

\bibitem{ttau}
Jos{\'e} Burillo, Brita Nucinkis, and Lawrence Reeves.
\newblock Irrational-slope versions of {T}hompson's groups {$T$} and {$V$}.
\newblock preprint, arXiv:2006.02401.

\bibitem{cfp}
J.~W. Cannon, W.~J. Floyd, and W.~R. Parry.
\newblock Introductory notes on {R}ichard {T}hompson's groups.
\newblock {\em Enseign. Math. (2)}, 42(3-4):215--256, 1996.

\bibitem{clearyirr}
Sean Cleary.
\newblock Regular subdivision in {$\zz[\frac{1+\sqrt 5}{2}]$}.
\newblock {\em Illinois J. Math.}, 44(3):453--464, 2000.

\bibitem{blakegd}
S.~Blake Fordham.
\newblock Minimal length elements of {T}hompson's group {$F$}.
\newblock {\em Geom. Dedicata}, 99:179--220, 2003.

\bibitem{higmansimple}
Graham Higman.
\newblock On infinite simple permutation groups.
\newblock {\em Publ. Math. Debrecen}, 3:221--226 (1955), 1954.

\bibitem{lodha_f23}
Yash Lodha.
\newblock Coherent actions by homeomorphisms on the real line or an interval.
\newblock {\em Israel J. Math.}, 235(1):183--212, 2020.

\bibitem{wladisdist}
Claire Wladis.
\newblock Thompson's group is distorted in the {T}hompson-{S}tein groups.
\newblock {\em Pacific J. Math.}, 250(2):473--485, 2011.

\end{thebibliography}
\bibliographystyle{plain}

\end{document}